\documentclass[oneside]{amsart}
\usepackage{enumerate,amssymb}
\newtheorem{thm}{Theorem}[section]
\newtheorem{coro}[thm]{Corollary}
\newtheorem{prop}[thm]{Proposition}
\newtheorem{lem}[thm]{Lemma}

\theoremstyle{definition}
\newtheorem{defn}[thm]{Definition}
\newtheorem{question}[thm]{Question}
\newcommand{\Rset}{\mathbb{R}}
\newcommand{\Nset}{\omega}
\newcommand{\Cset}{2^\Nset}
\newcommand{\CCset}{2^{<\Nset}}
\newcommand{\Pset}{\Nset^\Nset}
\newcommand{\UPset}{\Nset^{\uparrow\Nset}}

\newcommand{\eps}{\varepsilon}
\newcommand{\del}{\delta}
\newcommand{\subs}{\subseteq}
\newcommand{\sups}{\supseteq}
\newcommand{\clos}[1]{\overline{#1}}
\newcommand{\upto}{{\nearrow}}

\newcommand{\concat}{^\smallfrown\hspace{-0.3ex}}
\newcommand{\rest}{{\restriction}}
\renewcommand{\leq}{\leqslant}
\renewcommand{\geq}{\geqslant}

\newcommand{\abs}[1]{\lvert#1\rvert}
\newcommand{\seq}[1]{\langle#1\rangle}
\newcommand{\seqeps}{\seq{\eps_n}}
\newcommand{\Seqeps}{\seqeps\in(0,\infty)^\Nset}
\newcommand{\si}{$\sigma$\nobreakdash-}
\newcommand{\hm}{\mathcal H}
\newcommand{\uhm}{\overline{\mathcal H}{}}

\newcommand{\cyl}[1]{[\mkern-3mu[#1]\mkern-3mu]}
\newcommand{\whm}{\boldsymbol{\lambda}}
\newcommand{\upint}{\int^*}

\DeclareMathOperator{\id}{id}

\DeclareMathOperator{\hdim}{\dim_{\mathsf{H}}}
\DeclareMathOperator{\uhdim}{\overline{\dim}_{\mathsf{H}}}

\DeclareMathOperator{\diam}{diam}

\newcommand{\NN}{\mathcal N}
\newcommand{\NNs}{\mathcal N_\sigma}
\newcommand{\MM}{\mathcal M}
\newcommand{\EE}{\mathcal E}

\newcommand{\CC}{\mathcal C}
\newcommand{\smz}{${\boldsymbol{\mathsf{Smz}}}$}
\newcommand{\ssmz}{${\boldsymbol{\mathsf{Smz}}^\sharp}$}

\newcommand{\mc}[1]{\mathcal{#1}}

\newcommand{\CCC}{\boldsymbol{\mc C}}

\newcommand{\hnull}{$\boldsymbol{\mc H}$-{\rm null}}

\newcommand{\fmany}{\forall^\infty}

\newenvironment{enum}{\begin{enumerate}[\rm(i)]}{\end{enumerate}}
{\begin{list}{\textbullet}{\setlength{\labelwidth}{1ex}\setlength{\leftmargin}{1.1em}}}%
{\end{list}}

\newcommand{\Implies}{\ensuremath{\Rightarrow}}
\begin{document}
\title
[Strong measure zero and meager-additive sets]
{Strong measure zero and meager-additive sets\\ through the prism of fractal measures}
\author{Ond\v rej Zindulka}
\address
{Department of Mathematics\\
Faculty of Civil Engineering\\
Czech Technical University\\
Th\'akurova 7\\
160 00 Prague 6\\
Czech Republic}
\email{ondrej.zindulka@cvut.cz}
\urladdr{http://mat.fsv.cvut.cz/zindulka}
\thanks{The author was supported by the grant GA\v CR 15--082185
of the Grant Agency of the Czech Republic.
Work on this project was partially conducted during the author's sabbatical stay
at the Instituto de matem\'aticas, Unidad Morelia,
Universidad Nacional Auton\'oma de M\'exico supported by CONACyT grant no.~125108.
}
\begin{abstract}
We develop a theory of \emph{sharp measure zero} sets that parallels
Borel's \emph{strong measure zero}, and
prove a theorem analogous to Galvin-Myscielski-Solovay Theorem, namely
that a set of reals has sharp measure zero if and only if it is
meager-additive.
Some consequences:
A subset of $\Cset$ is meager-additive if and only if it is
$\EE$-additive; if $f:\Cset\to\Cset$ is continuous and $X$ is meager-additive,
then so is $f(X)$.
\end{abstract}
\maketitle
\section{Introduction}
\label{sec:intro}

99 years ago \'Emile Borel~\cite{MR1504785} conceived the notion of
\emph{strong measure zero}: by his definition, a metric space $X$ has
\emph{strong measure zero} (thereinafter \smz) if for any sequence $\seq{\eps_n}$
of positive numbers there is a cover $\{U_n\}$ of $X$
such that $\diam U_n\leq\eps_n$ for all $n$.
In the same paper, Borel conjectured that every \smz{} set of reals was countable.
This statement known as \emph{Borel Conjecture} attracted
a lot of attention.

\subsection*{Borel Conjecture}

It is well-known that Borel Conjecture is independent of ZFC, the usual axioms
of set theory.
The proof of consistency of its failure was settled by 1948
by Sierpi\'nski~\cite{SIERPINSKI}, who proved in 1928 that the Continuum Hypothesis yields
a counterexample, namely the Luzin set, and G\"odel~\cite{MR0002514},
who proved in 1948 the consistency of the Continuum Hypothesis.

The consistency of the Borel Conjecture remained open until 1976
when Laver proved in his ground-breaking paper~\cite{MR0422027}
Borel Conjecture to be indeed consistent with ZFC.

To complete the picture, Carlson~\cite{MR1139474} proved in 1993 that the Borel Conjecture
implies that every separable \smz{} metric space is countable.

Over time numerous characterizations of strong measure zero were discovered.
We are going to focus on (besides the definition itself) three such characterizations.

\subsection*{Hausdorff dimension}

One way to characterize \smz{} is via Hausdorff dimension:
It is almost obvious that a \smz{} space has Hausdorff
dimension zero. Since \smz{} is preserved by uniformly continuous
mappings, it follows that any uniformly continuous image of a \smz{}
space has Hausdorff dimension zero. It is not difficult to prove that
the latter property actually characterizes \smz{}.
The essence of this characterization can be traced back to Besicovitch
papers~\cite{MR1555386,MR1555389}.

\subsection*{Galvin's Game}

A characterization in terms of infinite games was recently published by
Galvin~\cite{MR3696064} who attributes it to Galvin, Mycielski and Solovay.
Consider the following game:
the playground is a subset $X$ of a \si compact metric
space. At the $n$-th inning, Player I chooses $\eps_n>0$
and Player II responds with a set $U_n\subs X$ such that
$\diam U_n\leq\eps_n$.
Player II wins if the sets $U_n$
form a cover of $X$, otherwise Player I wins. Denote this game $G(X)$.

\begin{thm}[{\cite{MR3696064}}]\label{Galvin}
Let $X$ be a subset of a \si compact metric space.
The set is \smz{} if and only if Player I does not have a winning strategy
in the game $G(X)$.
\end{thm}

\subsection*{Galvin--Mycielski--Solovay Theorem}

Confirming a Prikry's conjecture, Galvin, Mycielski and Solovay~\cite{GMS}
proved a rather surprising characterization of \smz{} subsets of the line:

\begin{thm}[\cite{GMS}]\label{GMS}
A set $X\subs\Rset$ is \smz{} if and only if $X+M\neq\Cset$
for each meager set $M\subs\Cset$.
\end{thm}

We will refer to this result as a \emph{Galvin--Mycielski--Solovay Theorem}.
Recently Kysiak \cite{kysiak} and Fremlin \cite{fremlin5} showed that an
analogous theorem holds for all \si compact metrizable groups.
The theorem was further investigated by
Hru\v s\'ak, Wohofsky and Zindulka~\cite{MR3453581}
and Hru\v s\'ak and Zapletal~\cite{MR3707641} who found, roughly speaking, that under the
Continuum Hypothesis Galvin--Mycielski--Solovay Theorem does not extend
beyond \si compact metrizable groups.

\medskip
In summary, we thus have four strikingly different descriptions of \smz{}:
\begin{itemize}
\item ``combinatorial'' --- the Borel's definition,
\item ``fractal'' --- by Hausdorff dimension of images,
\item ``game-theoretic'' --- by the Galvin's game
(restriction: subsets of \si compact spaces),
\item ``algebraic'' --- by the Galvin--Mycielski--Solovay Theorem (restriction:
subsets of \si compact metrizable groups).
\end{itemize}

\subsection*{Sharp measure zero}

Consider the characterization of \smz{} by Hausdorff dimension:
$X$ is \smz{} if and only if $\hdim f(X)=0$ for every uniformly continuous mapping.
One may, just out of curiosity, ask what happens when the Hausdorff dimension
is replaced with some other fractal dimension. Here we will consider the so called
\emph{upper Hausdorff dimension} $\uhdim$ introduced in~\cite{MR2957686}.
We will say a metric space has \emph{sharp measure zero} (thereinafter \ssmz)
if $\uhdim f(X)=0$ for every uniformly continuous mapping.

It turns out that \ssmz{} sets can be characterized by a property very
much like the Borel's definition of \smz{} and that properties of \ssmz{} sets
nicely parallel those of \smz{} sets.
In particular, \ssmz{} is characterized by a slight modification of the
Galvin's game.

One of the highlights of this section is the following improvement of
a theorem of Scheepers~\cite[Theorem 1]{MR1779763}):
a product of a \smz{} set and \ssmz{} set is \smz{}.

\subsection*{Meager-additive sets}

The Cantor set $\Cset$ with the coordinatewise addition
is a second countable compact topological group.
Provide $\Cset$ with the usual least difference metric.

Consider the following strengthening of the algebraic property of the
Galvin--Mycielski--Solovay Theorem:
say that a set $X\subs\Cset$ is \emph{meager-additive} if
$X+M$ is meager for every meager set $M\subs\Cset$. The notion generalizes
to other topological groups, and in particular to finite cartesian powers
of $\Cset$ and $\Rset$, in an obvious way.

Meager-additive sets in $\Cset$ have received a lot of attention.
They were investigated by many, most notably by
Bartoszy\'nski and Judah~\cite{MR1350295}, Pawlikowski~\cite{MR1380640}
and Shelah~\cite{MR1324470}.
Combinatorial properties of meager-additive sets described by
Pawlikowski~\cite{MR1380640} and Shelah~\cite{MR1324470}
allow to prove a rather surprising theorem that is one of
summits of the present paper.

\begin{thm}\label{thm:ex1}
A set $X\subs\Cset$ is \ssmz{} if and only if it is meager-additive.
\end{thm}

In summary, we thus have four descriptions of \ssmz{} that perfectly parallel
those of \smz{}:
\begin{itemize}
\item ``combinatorial'' --- a Borel-like definition,
cf.~Theorem~\ref{combUhnull},
\item ``fractal'' --- by upper Hausdorff measures,
cf.~Theorem \ref{basicUhnull},
\item ``game-theoretic'' --- by a Galvin-like game, this time without any restriction,
cf.~Theorem \ref{Galvins},
\item ``algebraic'' ---  by meager-additive sets (restriction:
subsets of $\Cset$ or Euclidean spaces and their finite powers),
cf.~\ref{cpowers} and~\ref{weiss22}.
\end{itemize}

Consequences include, for instance:
\begin{itemize}
  \item meager-additive sets are preserved by continuous mappings $f:\Cset\to\Cset$
  \item a product of a \smz{} and a meager-additive set is \smz{}
  \item meager-additive sets are universally meager (cf.~Proposition~\ref{umg})
\end{itemize}

Besides meager-additive sets, we also consider the following notion: a set
$X\subs\Cset$ is called \emph{$\EE$-additive} if
for every $F_\sigma$-set $E\subs\Cset$ of Haar measure zero the set
$X+E$ is contained in an $F_\sigma$-set of Haar measure zero.
We prove the following:

\begin{thm}\label{thm:ex2}
A set $X\subs\Cset$ is meager-additive if and only if it is $\EE$-additive.
\end{thm}
This theorem answers a question of Nowik and Weiss~\cite{MR1905154}.

\medskip

\smallskip

Some common notation used throughout the paper includes $\abs{A}$ for
the cardinality of a set $A$, $\Nset$ for the set of natural numbers,
$[\Nset]^\Nset$ for the collection of infinite
subsets of $\Nset$,
$\Pset$ for the family of all sequences of natural numbers,
and $\UPset$ for the family of nondecreasing
unbounded sequences of natural numbers.

\section{Strong measure zero via Hausdorff measure}
\label{sec:smz}

In this section we establish a few characterizations of strong measure zero
in terms of Hausdorff measures and dimensions based on a classical
Besicovitch result~\cite{MR1555386,MR1555389} and derive some consequences.

\subsection*{Hausdorff measure}

Before getting any further we need to review Hausdorff measure
and dimension. We set up the necessary definitions
and recall relevant facts.

Let $X$ be a space. If $A\subs X$, then $\diam A$ denotes the diameter of $A$.
A closed ball of radius $r$ centered at $x$ is denoted by $B(x,r)$.

A non-decreasing, right-continuous function $h:[0,\infty)\to[0,\infty)$
such that $h(0)=0$ and $h(r)>0$ if $r>0$ is called a \emph{gauge}.
The following is the common ordering of gauges, cf.~\cite{MR0281862}:
$$
  g\prec h\quad \overset{\mathrm{def}}{\equiv}
  \quad\lim_{r\to0+}\frac{h(r)}{g(r)}=0.
$$
In the case when $h(r)=r^s$ for some $s>0$, we write $g\prec s$ instead of
$g\prec h$.

Notice that for any sequence $\seq{h_n}$ of gauges there is
a gauge $h$ such that $h\prec h_n$ for all $n$.

If $\del>0$, a cover $\mc A$ of a set $E\subs X$ is termed a
\emph{$\del$-fine cover} if $\diam A\leq\del$ for all $A\in\mc A$.
If $h$ is a gauge,
the \emph{$h$-dimensional Hausdorff measure} $\hm^h(E)$ of
a set $E$ in a space $X$ is defined thus:
For each $\del>0$ set
$$
  \hm^h_\delta(E)=
  \inf\left\{\sum_{n\in\Nset}h(\diam E_n):
  \text{$\{E_n\}$ is a countable $\delta$-fine cover of $E$}\right\}
$$
and put
$
  \hm^h(E)=\sup_{\delta>0}\hm^h_\delta(E).
$

In the common case when $h(r)=r^s$ for some $s>0$, we write $\hm^s$ for
$\uhm^h$, and the same licence is used for other measures and set functions
arising from gauges.

Properties of Hausdorff measures are well-known. The following, including
the two propositions, can be found e.g.~in~\cite{MR0281862}.
The restriction of $\hm^h$ to Borel sets is a $G_\del$-regular Borel measure.
Recall that a sequence of sets $\seq{E_n:n\in\Nset}$ is termed a \emph{$\lambda$-cover}
of $E\subs X$ if every point of $E$ is contained in infinitely many $E_n$'s.
\begin{lem}\label{lambda}
$\hm^h(E)=0$ if and only if $E$ admits a countable
$\lambda$-cover $\seq{E_n}$ such that $\sum_{n\in\Nset}h(dE_n)<\infty$.
\end{lem}
\begin{lem}\label{lemHaus}
\begin{enum}
\item If $\hm^h(X)<\infty$ and $h\prec g$, then $\hm^g(X)=0$.
\item If $\hm^h(X)=0$, then there is $g\prec h$ such that $\hm^g(X)=0$.
\end{enum}
\end{lem}
We will also need a cartesian product inequality.
Given two metric spaces $X$ and $Y$ with
respective metrics $d_X$ and $d_Y$, provide the cartesian product $X\times Y$
with the maximum metric
\begin{equation}\label{maxmetric}
  d\bigl((x_1,y_1),(x_2,y_2)\bigr)=\max(d_X(x_1,x_2),d_Y(y_1,y_2)).
\end{equation}
A gauge $h$ satisfies the \emph{doubling condition} or $h$ is \emph{doubling}
if $\varlimsup_{r\to0}\frac{h(2r)}{h(r)}<\infty$.
\begin{lem}[{\cite{MR0318427,MR1362951}}]\label{howroyd}
Let $X,Y$ be metric spaces, $g$ a gauge and $h$ a doubling gauge. Then
$\hm^h(X)\,\hm^g(Y)\leq\hm^{hg}(X\times Y)$.
\end{lem}

The following lemma on Lipschitz images and its counterpart for uniformly
continuous mappings are well-known, see, e.g., \cite[Theorem 29]{MR0281862}.
\begin{lem}\label{lipschitz}
Let $f:(X,d_X)\to (Y,d_Y)$ be a mapping.
\begin{enum}
\item
If $f$ is uniformly continuous and a gauge $g$ is its modulus, i.e.,
\begin{equation}\label{lip2}
  d_Y(f(x),f(y))\leq g(d_X(x,y)),\quad x,y\in X,
\end{equation}
then $\hm^h(f(X))\leq \hm^{h{\circ}g}(X)$ for any gauge $h$.
\item
If $f$ Lipschitz with Lipschitz constant $L$,
then $\hm^s(f(X))\leq L^s\hm^s(X)$ for any $s>0$.
\end{enum}
\end{lem}
Recall that the \emph{Hausdorff dimension} of $X$ is defined by
$$
  \hdim X=\sup\{s>0:\hm^s(X)=\infty\}=\inf\{s>0:\hm^s(X)=0\}.
$$
Properties of Hausdorff dimension are well-known. In particular,
it follows from Lemma~\ref{lipschitz}(ii) that if $f:X\to Y$ is Lipschitz, then
$\hdim f(X)\leq\hdim X$.

Our first theorem provides a couple of characterizations of \smz{} spaces in terms of
Hausdorff measures and dimensions.
%
\begin{thm}\label{basicHnull}
Let $X$ be a metric space. The following are equivalent.
\begin{enum}
\item $X$ is \smz{},
\item $\hm^h(X)=0$ for each gauge $h$,
\item $\hdim f(X)=0$ for each uniformly continuous mapping $f$ on $X$,
\item $\hdim f(X,\rho)=0$ for each uniformly equivalent metric on $X$.
\end{enum}
\end{thm}
\begin{proof}
%
The equivalence of (i) and (ii), is due to Besicovitch~\cite{MR1555386,MR1555389}.

(ii)\Implies(iii)
Let $s>0$ be arbitrary. Let $f:X\to Y$ be uniformly continuous and let $g$
be the modulus of $f$. Define $h(x)=(g(x))^s$. By (ii) $\hm^h(X)=0$ and thus
Lemma~\ref{lipschitz}(i) yields $\hm^s(f(X))\leq\hm^h(X)=0$.
Since this holds for all $s>0$, we have $\hdim f(X)=0$.

(iii)\Implies(iv) is trivial.

(iv)\Implies(ii)
Denote by $d$ the metric of $X$.
Let $h$ be a gauge. Choose a strictly increasing, convex (and in particular subadditive)
gauge $g$ such that $g\prec h$.
The properties of $g$ ensure that $\rho(x,y)=g(d(x,y))$
is a uniformly equivalent metric on $X$. The identity
map $\id_X:(X,\rho)\to(X,d)$ is of course uniformly continuous and its modulus
is $g^{-1}$, the inverse of $g$. Hence by Lemma~\ref{lipschitz}(i)
$\hm^h(X,d)\leq\hm^{h\circ g^{-1}}(X,\rho)$.
Since $\hm^1(X,\rho)=0$ by assumption and $h\circ g^{-1}\succ1$
by the choice of $g$, we have $\hm^{h\circ g^{-1}}(X,\rho)=0$. Thus
$\hm^h(X,d)=0$, as required.
\end{proof}

Our next goal is to characterize \smz{} by behavior of cartesian products.
We need to recall first a few facts about the Cantor set.

\subsection*{Cantor set}
The set of all countable binary sequences is denoted by $\Cset$.
The set of all finite binary sequences is denoted by $\CCset$, i.e.,
$\CCset=\bigcup_{n\in\Nset}2^n=\{f:n\to2:n\in\Nset\}$.
For $p\in\CCset$ we denote $\cyl p=\{x\in\Cset:p\subs x\}$ the cone determined
by $p$. The family of all cones forms a basis for the topology of $\Cset$
and for $T\subs\CCset$ we let $\cyl T=\bigcup_{p\in T}\cyl p$.
It is well-known that this topology is second countable and compact.
It also obtains from the so called least difference metric:
For $x\neq y\in\Cset$,
set $n(x,y)=\min\{i\in\Nset:x(i)\neq y(i)\}$ and define $d(x,y)=2^{-n(x,y)}$.

The coordinatewise addition modulo $2$ makes $\Cset$ a compact topological group.
Routine proofs show that in this metric, $\hm^1$ coincides on Borel sets with
its Haar measure, i.e., the usual product measure on $\Cset$. In particular
$\hm^1(\Cset)=1$.

We consider the important \si ideal $\EE$ generated by closed null sets, i.e.,
the ideal of all sets that are contained in an $F_\sigma$ set of Haar measure zero.
\begin{lem}\label{EC}
\begin{enum}
\item For each $I\in[\Nset]^\Nset$, the set $C_I=\{x\in\Cset:x\rest I\equiv0\}$
is in $\EE$.
\item for each $h\prec 1$ there is $I\in[\Nset]^\Nset$ such that $\hm^h(C_I)>0$.
\end{enum}
\end{lem}
\begin{proof}
%
(i) Let $I\in[\Nset]^\Nset$. For each $n\in\Nset$,
the family $\{\cyl p:p\in C_I\rest n\}$ is obviously a $2^{-n}$-cover
of $C_I$ of cardinality $2^{\abs{n\setminus I}}$. Therefore
$\hm^1_{2^{-n}}(C_I)\leq2^{\abs{n\setminus I}}2^{-n} =2^{-\abs{n\cap I}}$.
Hence $\hm^1(C_I)\leq\lim_{n\to\infty}2^{-\abs{n\cap I}}=0$.

(ii)
$h\prec 1$ yields  $\frac{h(2^{-n})}{2^{-n}}\to\infty$. Therefore there is
$I\in[\Nset]^\Nset$ sparse enough to satisfy
$2^{\abs{n\cap I}}\leq\frac{h(2^{-n})}{2^{-n}}$,
i.e., $2^{-\abs{n\setminus I}}\leq h(2^{-n})$ for all $n\in\Nset$.
Consider the product measure on $C_I$ given as follows: If $p\in2^n$ and
$\cyl p\cap C_I\neq\emptyset$, put
$\lambda(\cyl p\cap C_I)=2^{-\abs{n\setminus I}}$.
Straightforward calculation shows that $h(dE)\geq\lambda(E)$ for each $E\subs C_I$.
Hence $\sum_n h(dE_n)\geq\sum_n\lambda(E_n)\geq\lambda(C_I)=1$
for each cover $\{E_n\}$ of $C_I$ and $\hm^h(C_I)\geq 1$ follows.
\end{proof}

\begin{thm}\label{prodHnull}
The following are equivalent.
\begin{enum}
\item $X$ is \smz{},
\item $\hm^h(X\times Y)=0$ for every gauge $h$ and every \si compact metric space $Y$
such that $\hm^h(Y)=0$,
\item $\hm^1(X\times E)=0$ for every $E\in\EE$,
\item $\hm^1(X\times C_I)=0$ for every $I\in[\Nset]^\Nset$.
\end{enum}
\end{thm}
\begin{proof}
(i)\Implies(ii):
Suppose $X$ is \smz{}. We may clearly suppose that $Y$ is compact.
Fix $\eta>0$.
Since $\hm^h(Y)=0$, for each $j\in\Nset$ there is a finite
family $\mc U_j$ of (open) sets such that $\sum_{U\in\mc U_j}h(\diam U)<2^{-j}\eta$.
We may also assume that $\diam U<\eta$ for all $U\in\mc U_j$.

Let $\eps_j=\min\{\diam U:U\in\mc U_j\}$.
Choose a cover $\{V_j\}$ of $X$ such that $\diam V_j\leq\eps_j$ and define
$$
  \mc W=\{V_j\times U:j\in\Nset,\,U\in\mc U_j\}.
$$
It is obvious that $\mc W$ is a cover of $X\times Y$. Since
$\diam(V_j\times U)=\diam U$ for all $j$ and $U\in\mc U_j$ by the choice
of $\eps_j$, we have
$$
  \sum_{W\in\mc W}h(\diam W)=
  \sum_{j\in\Nset}\sum_{U\in\mc U_j}h(\diam U)<
  \sum_{j\in\Nset}2^{-j}\eta=2\eta.
$$
Therefore $\hm_\eta^h(X\times Y)<2\eta$, which is enough for
$\hm^h(X\times Y)=0$, as $\eta$ was arbitrary.

(ii)\Implies(iii)\Implies(iv) is trivial.

(iv)\Implies(i):
Suppose $X$ is not \smz{}. We will show that $\hm^1(X\times C_I)>0$
for some $I\in[\Nset]^\Nset$.
By assumption there is a gauge $h$ such that $\hm^h(X)>0$.
\emph{Mutatis mutandis} we may assume $h$ be concave and $h(r)\geq\sqrt r$.
In particular, by concavity of $h$ the function $g(r)=r/h(r)$ is increasing and
$h(r)\geq\sqrt r$ yields $\lim_{r\to0}g(r)=0$, i.e., $g$ is a gauge, and
$g\prec1$.

Use
Lemma~\ref{EC}(ii) to find $I\in[\Nset]^\Nset$ such that $\hm^g(C_I)>0$.
Since $h$, being concave, is a doubling gauge, we may
apply Lemma~\ref{howroyd}:
$$
  \hm^1(X\times C_I)=\hm^{h\cdot g}(X\times C_I)\geq\hm^h(X)\cdot\hm^g(C_I)>0.
  \qedhere
$$
\end{proof}
\begin{coro}\label{hdimx}
If $X$ is \smz{}, then $\hdim X\times Y=\hdim Y$ for every \si compact metric space $Y$.
\end{coro}

\section{Sharp measure zero}
\label{sec:ssmz}

In this section we develop elementary theory of a notion a bit stronger than that of
strong measure zero. The following definition is inspired by
Theorem~\ref{basicHnull}(iii).
\begin{defn}
A metric space $X$ has \emph{sharp measure zero} (\ssmz{}) if
for every uniformly continuous mapping $fX\to Y$ into a complete metric space $Y$
there is a \si compact set $K\subs Y$ such that $f(X)\subs K$ and $\hdim K=0$.
\end{defn}

It is obvious that \ssmz{} is a \si additive property and that it is preserved
by uniformly continuous maps:
\begin{prop}\label{trivUH}
\begin{enum}
\item
If $X$ is a metric space, then
the family of all \ssmz{} subsets of $X$ forms a \si ideal.
\item
If $X$ is \ssmz{} and $f:X\to Y$ is a uniformly continuous mapping, then
$f(X)$ is \ssmz{}.
\end{enum}
\end{prop}

In this section we provide a few characterizations and describe a few properties
of sharp measure zero.

\subsection*{Upper Hausdorff measure}

It turns out that sharp measure zero can be described in terms of
a fractal measure very similar to Hausdorff measure.
It is defined thus: let $h$ be a gauge. For each $\delta>0$ set
$$
  \uhm^h_\delta(E)=
  \inf\left\{\sum_{n=0}^N h(\diam E_n):
  \text{$\{E_n:n\leq N\}$ is a \emph{finite} $\delta$-fine cover of $E$}\right\}.
$$
Then put
$
  \uhm_0^h(E)=\sup_{\delta>0}\uhm^h_\delta(E).
$
The only difference from $\hm^h$ is that only finite covers are taken in account.
It is easy to check that $\uhm_0^h$ is finitely subadditive, but unfortunately
it is not a measure, since it need not be \si additive.
To overcome this difficulty we apply to $\uhm_0^h$ the operation
known as Munroe's \emph{Method I construction} (cf.~\cite{MR0053186} or \cite{MR0281862}):
$$
  \uhm^h(E)=\inf\Bigl\{\sum_{n\in\Nset}\uhm^h_0(E_n):
  E\subs\bigcup_{n\in\Nset}E_n\Bigr\}.
$$
Thus defined set function is indeed an outer measure whose restriction to
Borel sets is a Borel measure.
\begin{defn}
The measure $\uhm^h$  is called the \emph{$h$-dimensional upper Hausdorff measure}.
\end{defn}
We list some properties of $\uhm_0^h$ and $\uhm^h$.
Some of them will be utilized below and some are provided just to shed more light on the
notion of upper Hausdorff measure.
The straightforward proofs are omitted.
Denote $\NNs(\uhm_0^h)$ the family of countable unions of sets $E$ with $\uhm_0^h(E)=0$.
We also write $E_n\upto E$ to denote that $\seq{E_n}$ is an increasing sequence
of sets with union $E$.
\begin{lem}\label{lem1}
Let $h$ be a gauge and $E$ a set in a metric space.
\begin{enum}
\item If $\uhm_0^h(E)<\infty$, then $E$ is totally bounded.
\item $\uhm_0^h(E)=\uhm_0^h(\clos E)$.
\item $\uhm_0^h(E)=\hm^h(E)$ if $E$ is compact.
\item If $X$ is complete, $E\subs X$ and $E\in\NNs(\uhm_0^h)$, then there is a \si compact set
$K\sups E$ such that $\hm^h(K)=0$.
\item If $X$ is complete and $E\subs X$, then
$\uhm^h(E)=\inf\{\hm^h(K):\text{$K\sups E$ is \si compact}\}$.
\item In particular $\uhm^h(E)=\hm^h(E)$ if $E$ is \si compact.
\item If $g\prec h$ and $\uhm^g(E)<\infty$, then $E\in\NNs(\uhm_0^h)$.
\item If $E\in\NNs(\uhm_0^h)$, then there is a sequence $E_n\upto E$ such that
$\uhm_0^g(E_n)=0$ for all $n$.
\item If $\uhm^g(E)<s$, then there is a sequence $E_n\upto X$ such that
$\sup\uhm_0^g(E_n)<s$.
\end{enum}
\end{lem}
We will also need lemmas that parallel Lemmas~\ref{howroyd} and~\ref{lipschitz}.
As to the proofs, Lemma~\ref{uhowroyd} is proved in the Appendix
and Lemma~\ref{ulipschitz} is proved exactly the same way as Lemma~\ref{lipschitz}.

\begin{lem}\label{uhowroyd}
Let $X,Y$ be metric spaces and $g$ a gauge and $h$ a doubling gauge. Then
$\hm^h(X)\,\uhm^g(Y)\leq\uhm^{hg}(X\times Y)$.
\end{lem}
\begin{lem}\label{ulipschitz}
Let $f:(X,d_X)\to (Y,d_Y)$ be a mapping.
\begin{enum}
\item
If $f$ is uniformly continuous and a gauge $g$ is its modulus, i.e.,
satisfies~\eqref{lip2},
%
%
then $\uhm^h(f(X))\leq \uhm^{h{\circ}g}(X)$ for any gauge $h$.
\item
If $f$ Lipschitz with Lipschitz constant $L$,
then $\uhm^s(f(X))\leq L^s\uhm^s(X)$ for any $s>0$.
\end{enum}
\end{lem}
The corresponding \emph{upper Hausdorff dimension} of $X$,
introduced in~\cite{MR2957686}, is defined by
$$
  \uhdim X=\sup\{s>0:\uhm^s(X)=\infty\}=\inf\{s>0:\uhm^s(X)=0\}.
$$
It is clear that $\hdim X\leq\uhdim X$. The inequality may be strict, cf.~examples
in~\cite[Section 2]{MR2957686} and~\cite[Example 4.2]{MR3114775}.

It follows from Lemma~\ref{lem1}(v) that if $X$ is a complete metric space and $E\subs X$,
then $\uhdim E=\inf\{\hdim K:K\supseteq E\text{ is \si compact}\}$.
In particular, if $X$ is \si compact, then $\hdim X=\uhdim X$.

It follows from Lemma~\ref{ulipschitz}(ii) that if $f:X\to Y$
is Lipschitz, then $\uhdim f(X)\leq\uhdim X$.

\medskip
We now establish the \ssmz{} counterpart
Theorem~\ref{basicHnull}.
\begin{thm}\label{basicUhnull}
The following are equivalent.
\begin{enum}
\item $X$ is \ssmz{},
\item $\uhm^h(X)=0$ for each gauge $h$,
\item $\uhdim f(X)=0$ for each uniformly continuous mapping $f$ on $X$,
\item $\uhdim f(X,\rho)=0$ for each uniformly equivalent metric on $X$.
\end{enum}
\end{thm}
\begin{proof}
(i)$\Leftrightarrow$(iii) follows at once from Lemma~\ref{lem1}.
(iii)\Implies(iv) is trivial.
(iv)\Implies(ii)\Implies(iii) goes exactly the same way as that
in Theorem~\ref{basicHnull},
one has to employ Lemma~\ref{ulipschitz} instead of Lemma~\ref{lipschitz}.
\end{proof}

Our next goal is to describe \ssmz{} in terms of covers. The characterization
parallels Borel's original definition of \smz{}.

\begin{defn}
Let $\seq{U_n}$ be a sequence of sets in $X$.
Recall that $\seq{U_n}$ is called
a \emph{$\gamma$-cover} if each $x\in X$ belongs to all but finitely
many $U_n$.

Recall that $\seq{U_n}$ is called \emph{$\gamma$-groupable cover}
if there is a partition $\Nset=I_0\cup I_1\cup I_2\cup\dots$
into consecutive finite intervals (i.e.~$I_{j+1}$ is on the right of $I_j$ for all $j$)
such that the sequence $\seq{\bigcup_{n\in I_j}U_n:j\in\Nset}$ is a $\gamma$-cover.
The partition $\seq{I_j}$ will be occasionally called \emph{witnessing}
and the finite families $\{U_n:n\in I_j\}$
will be occasionally called \emph{witnessing families}.
\end{defn}

The following is a counterpart of Lemma~\ref{lambda}.
\begin{lem}\label{gammagr}
$E\in\NNs(\uhm_0^h)$ if and only if $E$ has a $\gamma$-groupable
cover $\seq{U_n}$ such that $\sum_{n\in\Nset}h(\diam U_n)<\infty$.
\end{lem}
\begin{proof}
\Implies{} Let $E_n\upto E$, $\uhm^g_0(E_n)=0$.
For each $n$ let $\mc G_n$ be a finite cover of $E_n$ such that
$\sum_{G\in\mc G_n}g(\diam G)<2^{-n}$. The required cover is $\mc G=\bigcup_n\mc G_n$,
with $\mc G_n$ the witnessing families.

$\Leftarrow$ Let $\mc G_j$ be the witnessing families.
Put $E_k=\bigcap_{j\geq k}\bigcup\mc G_j$. Fix $k$. The set $E_k$ is covered
by each $\mc G_j$, $j\geq k$, and $\sum_{G\in\mc G_j}g(\diam G)$ is as
small as needed if $j$ is large enough. Hence $\uhm^g_0(E_k)=0$.
\end{proof}

We often deal with sequences of positive real numbers. Instead
of writing always ``let $\seqeps$ be a sequence of positive numbers''
we briefly write ``let $\Seqeps$''.

Let $X$ be a metric space and let $\seq{U_n:n\in\Nset}$ be a sequence of
subsets of $X$. Say that $\seq{U_n:n\in\Nset}$ is \emph{$\seqeps$-fine}
if $\diam U_n\leq\eps_n$ holds for all $n$.

\begin{lem}\label{seqep}
\begin{enum}
\item
For each $\Seqeps$ there exists a gauge $h$ such that if $X$ is a metric space and
$\uhm^h(X)=0$, then $X$ admits an $\seqeps$-fine $\gamma$-groupable cover.
\item
For each gauge $h$ there exists $\Seqeps$ such that if $X$ is a metric space
that admits an $\seqeps$-fine $\gamma$-groupable cover, then $\uhm^h(X)=0$.
\end{enum}
\end{lem}
\begin{proof}
(i)
Let$\seq{\eps_n}\in(0,\infty)^\Nset$.
Choose gauges $g,h$ such that
$h(\eps_n)>\frac1n$ for all $n\in\Nset$ and $g\prec h$.
Suppose $X$ is a metric space such that $\uhm^g(X)=0$. Then $X\in\NNs(\uhm_0^h)$
by Lemma~\ref{lem1}(vii). By Lemma~\ref{gammagr} there is a
$\gamma$-groupable cover $\{G_n\}$ such that $\sum_nh(\diam G_n)<\infty$.
Let $\{I_j:j\in\Nset\}$ be the witnessing partition and $\mc G_j=\{G_n:n\in I_j\}$
the witnessing groups.

We plan to permute the cover so that diameters decrease.
One obstacle is that some of them may be $0$.
Another one is that permutation may break down the witnessing groups.
We have to work around these difficulties.

For each $n$ choose $\del_n>\diam G_n$ so that $\sum_nh(\del_n)<\infty$.
Then recursively choose an increasing sequence $\seq{j_k}$ such that
for all $k\in\Nset$
\begin{enumerate}
\item[(a)] $\sum\{h(\del_n):n\in I_{j_k}\}<2^{-k}$,
\item[(b)] $\max\{\del_n:n\in I_{j_{k+1}}\}<\min\{\del_n:n\in I_{j_{k}}\}$
(this is possible since $\del_n$'s are positive).
\end{enumerate}
Let $I=\bigcup_{k\in\Nset}I_{j_k}$.
Permute $G_n$'s within each $\mc G_{j_k}$ so that $\del_n$ does not increase
as $n$ increases. Together with (b) this ensures that
the sequence $\seq{\del_n:n\in I}$ is nonincreasing.
For each $i\in\Nset$ let $i^*\in I$ be the unique index such that $i=\abs{I\cap i^*}$
and define $H_i=G_{i^*}$.
It follows, with the aid of (a)
and the definition of $h$, that for all $n\in I$
\begin{align*}
  h(\diam H_i)=h(\diam G_{i^*})&\leq h(\del_{i^*})
  \leq\frac{1}{i}\sum\{h(\del_m):m\in I,m\leq i^*\}\\
  &\leq\frac{1}{i}\sum\{h(\del_m):m\in I\}
  \leq\frac{1}{i}<h(\eps_i)
\end{align*}
and thus $\diam H_i\leq\eps_i$, i.e., $\seq{H_i}$ is an $\seq{\eps_i}$-fine sequence.
Moreover, the groups $\mc G_{j_k}:k\in\Nset$ witness that $\seq{H_i}$
is a $\gamma$-groupable cover.

(ii)
Let $h$ be a gauge. Choose $\eps_n<\del$ to satisfy
$\sum_nh(\eps_n)<\infty$. If $X$ is a metric space admitting a $\seq{\eps_n}$-fine
$\gamma$-groupable cover $\seq{G_n}$, then
$\sum_nh(\diam G_n)\leq\sum_nh(\eps_n)<\infty$.
By Lemma~\ref{seqep} $\uhm^h(X)=0$.
\end{proof}
The Borel-like definition of \ssmz{} now follows at once from the above lemma.
\begin{thm}\label{combUhnull}
Let $X$ be a separable metric space.
$X$ is \ssmz{} if and only if for each $\Seqeps$, $X$ has
an $\seqeps$-fine $\gamma$-groupable cover.
\end{thm}

Our next goal is to set up a counterpart to Theorem~\ref{prodHnull}.

\begin{thm}\label{prodUhnull}
The following are equivalent.
\begin{enum}
\item $X$ is \ssmz,
\item for each gauge $h$, $Y\in\NNs(\uhm_0^h)$ and each complete space $Z\sups X$
there is a \si compact $F$, $X\subs F\subs Z$,
such that $\uhm^h(F\times Y)=0$,
\item $\uhm^h(X\times Y)=0$ for each gauge $h$ and $Y\in\NNs(\uhm_0^h)$,
\item $\uhm^1(X\times E)=0$ for each $E\in\EE$,
\item $\uhm^1(X\times C)=0$ for each $C\in\CC$.
\end{enum}
\end{thm}
\begin{proof}
The proof is similar to that of Theorem~\ref{prodHnull}. The only nontrivial implications
are
(i)\Implies(ii) and (v)\Implies(i).

(i)\Implies(ii):
Let $Z\sups X$ be a complete metric space.
Suppose $X$ is \ssmz{}. By Lemma~\ref{lem1}, $X$ is contained in
a \si compact set $K\subs Z$. Let $h$ be a gauge and $Y\in\NNs(\uhm_0^h)$.
Lemma~\ref{gammagr} yields a $\gamma$-groupable cover $\mc U$ of $Y$ such that
$\sum_{U\in\mc U}h(\diam U)<\infty$.
Denote by $\mc U_j$ the witnessing families.
Let $\eps_j=\min\{\diam U:U\in\mc U_j\}$.
Using Theorem~\ref{combUhnull} choose a $\gamma$-groupable cover $\{V_j\}$ of $X$
such that $\diam V_j\leq\eps_j$.  We may assume that each $V_j$ is a closed subset of $Z$.
Denote by $\mc V_k$ the witnessing families.
Define
\begin{align*}
  \mc W&=\{V_j\times U:j\in\Nset,\,U\in\mc U_j\},\\
  F&=K\cap\bigcup_{i\in\Nset}\bigcap_{k\geq i}\bigcup\mc V_k.
\end{align*}
The set $F\subs Z$ is clearly an $F_\sigma$ subset of $K$ and is thus \si compact.
It is easy to check that $\mc W$ is a $\gamma$-groupable cover of $F\times Y$.
Since $\diam(V_j\times U)=\diam U$ for all $j$ and $U\in\mc U_j$ by the choice
of $\eps_j$, we have
$
  \sum_{W\in\mc W}h(\diam W)=
  \sum_{U\in\mc U}h(\diam U)<\infty.
$
Using Lemma~\ref{gammagr} it follows that $F\times Y\in\NNs(\uhm_0^h)$ and
in particular $\uhm^h(X\times Y)=0$.

(v)\Implies(i):
Suppose $X$ is not \ssmz. We will show that $\uhm^1(X\times C)>0$
for some $C\in\CC$.
By assumption there is a gauge $h$ such that $\uhm^h(X)>0$. As well as in the proof of
Theorem~\ref{prodHnull} suppose $h$ is concave, hence doubling, and
find a gauge $g\prec1$ such that $g(r)h(r)=r$.
Then use Lemma~\ref{EC}(ii) to find $I\in[\Nset]^\Nset$ such that $\hm^g(C_I)>0$
and apply Lemma~\ref{uhowroyd}:
$$
  \uhm^1(X\times C)=\uhm^{h\cdot g}(X\times C)\geq\uhm^h(X)\cdot\hm^g(C)>0.
  \qedhere
$$
\end{proof}
\begin{coro}\label{uhdimx}
If $X$ is \ssmz{} then $\uhdim X\times Y=\uhdim Y$ for every metric space $Y$.
In particular, $\uhdim X\times Y=\hdim Y$ if $Y$ is \si compact.
\end{coro}

\subsection*{Products of \smz{} and \ssmz{} sets}
It is well known that a product of two \smz{} sets need not be \smz{}.
Thus the product of two \hnull{} sets need not be \hnull{}.
But if one of the factors is \ssmz{}, the product is \hnull{}:
\begin{thm}\label{productHUH}
\begin{enum}
\item If $X$ and $Y$ are \ssmz{}, then $X\times Y$ is \ssmz{}.
\item If $X$ is \smz{} and $Y$ is \ssmz{}, then $X\times Y$ is \smz{}.
\end{enum}
\end{thm}
\begin{proof}
Suppose $Y$ is \ssmz{}. By Theorem~\ref{basicUhnull}(ii)
and Lemma~\ref{lem1}(vii), $Y\in\NNs(\uhm_0^h)$ for all gauges $h$.

(i) If $X$ is \ssmz{}, then Theorem~\ref{prodUhnull}(iii) yields
$\uhm^h(X\times Y)=0$ for all gauges $h$, which is by
Theorem~\ref{basicUhnull}(ii) enough.

(ii) If $X$ is \smz{}, then Lemma~\ref{lem1} and Theorem~\ref{prodHnull}(ii)
yield $\hm^h(X\times Y)=0$ for all gauges $h$, which is by
Theorem~\ref{basicHnull}(ii) enough.
\end{proof}

\section{The Galvin's game}
\label{sec:galvin}

As already discussed in the introduction, Galvin~\cite{MR3696064}
succeeded to characterize \smz{} sets in \si compact metric  spaces
in terms of a game, cf.~Theorem~\ref{Galvin}.

We consider a similar game and prove a counterpart of Galvin's theorem
for \ssmz{} sets. Note the striking similarity with the Galvin's game.

\begin{defn}
Let $X$ be a metric space. The game $G^\sharp(X)$ is played as follows:
At the $n$-th inning, Player I chooses $\eps_n>0$
and Player II responds with a set $U_n\subs X$ such that
$\diam U_n\leq\eps_n$.
Player II wins if the sequence of sets $\seq{U_n}$
forms a $\gamma$-groupable cover of $X$, otherwise Player I wins.
\end{defn}

\begin{thm}\label{Galvins}
A metric space $X$ is \ssmz{} if and only if Player I does not have
a winning strategy in $G^\sharp(X)$.
\end{thm}
Note that, unlike in Galvin's theorem, $X$ is not \emph{a priori} supposed to be
a subset of a \si compact space.
\begin{proof}
The backwards implication is trivial: if $X$ is not \ssmz{}, then
by Theorem~\ref{combUhnull} there is $\Seqeps$ such that $X$ has
no $\seqeps$-fine $\gamma$-groupable cover. The winning strategy for Player I
is of course to play $\eps_n$ at the $n$-th inning.

For the forward implication we modify the Galvin's proof.
Suppose that $X$ is \ssmz{}. By Theorem~\ref{basicUhnull}(ii)
$\uhm^h(X)=0$ for any gauge $h$ and thus  Lemma~\ref{lem1}(i) yields an
increasing sequence $F_n\upto X$ of totally bounded sets.

Let $\sigma$ be a strategy for Player I. We will show that $\sigma$ is not winning.

Recall~\cite[Lemma 1]{MR3696064}: \emph{If $F$ is totally bounded, then
for every $\del>$ there is a finite $\del$-fine collection $\mc B$ of sets,
such that every subset of $F$ of diameter at most $\del/3$ is contained in
some $B\in\mc B$.}

Using this fact, build recursively $\del_n$ and $\mc B_n$ as follows:
\begin{enumerate}
\item[(a)]
$\del_n=\inf\{\sigma(B_0,B_1,\dots,B_{n-1}):B_i\in\mc B_i,i<n\}$,
\item[(b)]
$\mc B_n$ is a finite $\del_n$-fine collection of subsets of $F_n$,
\item[(c)]
if $A\subs F_n$ and $\diam A\leq\frac{\del_n}{3}$, then $A\subs B$ for some
$B\in\mc B_n$.
\end{enumerate}
Since $X$ is \ssmz{}, there is a $\seq{\del_n/3}$-fine $\gamma$-groupable cover
$\seq{A_n}$ of $X$. By (b), for each $n\in\Nset$ we may choose $B_n\in\mc B_n$
such that $A_n\cap F_n\subs B_n$.
Put $\eps_n=\sigma(B_0,B_1,\dots,B_{n-1})$.
Since $\diam B_n\leq\eps_n$, the sequence $\seq{\eps_0,B_1,\eps_2,B_2,\dots}$
is played according to the strategy $\sigma$.

The sequence $\seq{B_n}$ is clearly $\seqeps$-fine. We claim
that it is also a $\gamma$-groupable cover of $X$.
We know that $\seq{A_n}$ is a $\gamma$-groupable cover. Let
$\seq{I_j}$ be the witnessing partition of $\Nset$.
Fix $x\in X$.
We have $\fmany j\ \exists k\in I_j\ x\in A_k$ and since $F_n\upto X$,
also $\fmany j\ \forall k\in I_j\ x\in F_k$.
Therefore $\fmany j\ \exists k\in I_j\ x\in A_k\cap F_k\subs B_k$.
Thus the partition $\seq{I_j}$ is also witnessing that $\seq{B_n}$ is
a $\gamma$-groupable cover of $X$.

Consequently, $\sigma$ is not a winning strategy.
\end{proof}

\section{\ssmz{}-sets vs.~$\MM$-additive and $\EE$-additive sets}
\label{sec:meager}

In this section we look closer at \ssmz{} subsets of the Cantor set $\Cset$.
Inspired by the Galvin--Mycielski--Solovay Theorem~\ref{GMS}, we
prove that \ssmz{} sets in $\Cset$ are meager-additive and vice versa.
Recall that $\MM$ denotes the ideal of meager sets.

A set $X\subs\Cset$ is called \emph{$\MM$-additive} (or \emph{meager-additive})
if $\forall M\in\MM\ X+M\in\MM$.
We also define a seemingly stronger notion: call $X$ \emph{sharply $\MM$-additive}
if $\forall M\in\MM\ \exists F\sups X\text{ \si compact } F+M\in\MM$.

\begin{thm}\label{mainME1}
For any set $X\subs\Cset$, the following are equivalent.
\begin{enum}
\item $X$ is \ssmz{}, \label{uhnull}
\item $X$ is $\MM$-additive, \label{Madd}
\item $X$ is sharply $\MM$-additive, \label{sMadd}
\item $\forall M\in\MM\ \exists F\sups X\text{ \si compact } F+M\neq\Cset$.
\label{sharpN}
\end{enum}
\end{thm}

Recall that $\EE$ is the ideal of Haar null $F_\sigma$-sets in $\Cset$.
We consider also $\EE$-additive and sharply $\EE$-additive sets.
A set $X\subs\Cset$ is called \emph{$\EE$-additive} if $\forall M\in\EE\ X+M\in\EE$
and \emph{sharply $\EE$-additive} if
$\forall M\in\EE\ \exists F\sups X\text{ \si compact } F+M\in\EE$.

\begin{thm}\label{mainME2}
For any set $X\subs\Cset$, the following are equivalent.
\begin{enum}
\item $X$ is \ssmz{}, \label{uhnull}
\item $X$ is $\EE$-additive, \label{Eadd}
\item $X$ is sharply $\EE$-additive. \label{sharpE}
\end{enum}
\end{thm}

%
\begin{proof}
We shall prove now \ref{mainME1}\eqref{uhnull}$\Implies$\ref{mainME2}\eqref{Eadd}
and
\ref{mainME2}\eqref{sharpE}$\Implies$\ref{mainME1}\eqref{sharpN}%
$\Implies$\ref{mainME1}\eqref{sMadd}%
$\Implies$\ref{mainME1}\eqref{Madd}.
The remaining implications
\ref{mainME1}\eqref{Madd}$\Implies$\ref{mainME1}\eqref{uhnull} and
\ref{mainME2}\eqref{Eadd}$\Implies$\ref{mainME2}\eqref{sharpE}
are subject to standalone Propositions~\ref{MeShelah} and~\ref{EisEsharp}.

\smallskip
\ref{mainME1}\eqref{uhnull}$\Implies$\ref{mainME2}\eqref{Eadd}:
Assume that $X$ is \ssmz{}. Let $E\in\EE$.
By Theorem~\ref{prodUhnull}, $X\times E\in\NNs(\uhm_0^1)$.
Since the mapping $(x,y)\mapsto x+y$ is Lipschitz, Lemma~\ref{lipschitz}(ii) yields
$X+E\in\NNs(\uhm_0^1)=\EE$.

\smallskip
\ref{mainME2}\eqref{sharpE}$\Implies$\ref{mainME1}\eqref{sharpN}:
Denote by $\NN$ the ideal of Haar null sets in $\Cset$.
We employ a theorem of Pawlikowski~\cite{MR1380640}
(or see also~\cite[Theorem 8.1.19]{MR1350295}):
\emph{For each $M\in\MM$ there exists $E\in\EE$ such that for each $Y\subs\Cset$,
if $Y+E\in\NN$, then $Y+M\neq\Cset$.}

Suppose $X$ is sharply $\EE$-additive. Let $M\in\MM$.
Let $E\in\EE$ be the set guaranteed by the Pawlikowski's theorem.
Since $X$ is sharply $\EE$-additive, there is $F\sups X$
\si compact such that $F+E\in\EE\subs\NN$.
Therefore $F+M\neq\Cset$.
Thus $X$ is sharply null.

\smallskip
\ref{mainME1}\eqref{sharpN}$\Implies$\ref{mainME1}\eqref{sMadd}:
Suppose $X$ is sharply null and let $M\in\MM$.
We may assume that $M$ is \si compact.
Let $Q\subs\Cset$ be a countable dense set.
Clearly $Q+M$ is meager.
Therefore there is $F\sups X$ such that $Q+M+F\neq\Cset$.
Choose $z\notin Q+M+F$. Then, for all $q\in Q$, $z\notin q+M+F$, i.e.,
$z+q\notin M+F$. Therefore $(M+F)\cap(Q+z)=\emptyset$.
Since $Q$ is dense, so is $Q+z$.
Therefore the complement of $F+M$ is dense.

Since $F+M$ is a continuous image of a \si compact set $F\times M$,
it is \si compact as well.
Since it has a dense complement, it is meager by
the Baire category theorem.

\smallskip
\ref{mainME1}\eqref{sMadd}$\Implies$\ref{mainME1}\eqref{Madd} is obvious.
\end{proof}
In order to prove that every $\MM$-additive set is \ssmz{} we need
a Shelah's~\cite{MR1324470} (or see \cite[Theorem 2.7.17]{MR1350295})
characterization of $\MM$-additive sets:
\begin{lem}[{\cite{MR1324470}}]\label{ShelahM}
$X$ is $\MM$-additive if and only if
\begin{multline*}
  \forall f\in\UPset\,\,\exists g\in\Pset\,\,\exists y\in\Cset\,\,
  \forall x\in X\,\,\forall^\infty n\,\,\exists k\\
  g(n)\leq f(k)<f(k+1)\leq g(n+1) \&\
  x\rest [f(k),f(k+1))=y\rest [f(k),f(k+1)).
\end{multline*}
\end{lem}
\begin{prop}\label{MeShelah}
If $X\subs\Cset$ is $\MM$-additive, then $X$ is \ssmz.
\end{prop}
\begin{proof}
Let $X\subs\Cset$ be $\MM$-additive. Let $h$ be a gauge.
We have to show that $\uhm^h(X)=0$.
Define recursively $f\in\UPset$
to satisfy
$$
  2^{f(k)}\cdot h\bigl(2^{-f(k+1)}\bigr)\leq 2^{-k},\quad k\in\Nset.
$$
By Lemma~\ref{ShelahM} there is $g\in\Pset$ and $y\in\Cset$ such that
\begin{multline}\label{ShelahM2}
  \forall x\in X\,\,\forall^\infty n\,\,\exists k\\
  g(n)\leq f(k)<g(n+1)\ \&\
  x\rest [f(k),f(k+1))=y\rest [f(k),f(k+1)).
\end{multline}
Recall that if $p\in\CCset$ then $\cyl p$ denotes the cone $\{f\in\Cset:p\subs f\}$.
Define
\begin{alignat*}{3}
  &\mc B_k&&=
  \bigl\{
  \cyl{p\concat y\rest [f(k),f(k+1))}:p\in 2^{f(k)}
  \bigr\},\qquad
  && k\in\Nset,\\
  &\mc G_n&&=\bigcup
  \bigl\{\mc B_k:g(n)\leq f(k)<g(n+1)\bigr\},
  && n\in\Nset,\\
 & \mc B&&=\bigcup_{k\in\Nset}\mc B_k=\bigcup_{n\in\Nset}\mc G_n.
\end{alignat*}
With this notation~\eqref{ShelahM2} reads
\begin{equation}\label{ShelahM22}
  \forall x\in X\,\,\forall^\infty n\,\,\exists G\in\mc G_n\,\,
  x\in G.
\end{equation}
Since each of the families $\mc G_n$ is finite,
it follows that $\mc G_n$'s witness that $\mc B$ is a $\gamma$-groupable
cover of $X$.
Using Lemma~\ref{gammagr} it remains to show
that the Hausdorff sum $\sum_{B\in\mc B}h(\diam B)$ is finite.
Since $\abs{\mc B_k}=2^{f(k)}$ and
$\diam B=2^{-f(k+1)}$ for all $k$ and all $B\in\mc B_k$, we have
$$
  \sum_{B\in\mc B}h(\diam B)=
  \sum_{k\in\Nset}\sum_{B\in\mc B_k}h(\diam B)=
  \sum_{k\in\Nset}2^{f(k)}\cdot h(2^{-f(k+1)})
  \leq\sum_{k\in\Nset}2^{-k}<\infty.
  \qedhere
$$
\end{proof}

In order to prove that every $\EE$-additive set is sharply $\EE$-additive, we
employ a combinatorial description of closed null sets
given by Bartoszy\'nski and Shelah~\cite{MR1186905},
see also~\cite[2.6.A]{MR1350295}.
For $f\in\UPset$ let
$$
  \CCC_f=\left\{\seq{F_n}:\forall n\in\Nset\left(F_n\subs2^{[f(n),f(n+1))}\ \&\
         \frac{\abs{F_n}}{2^{f(n+1)-f(n)}}\leq\frac{1}{2^n}\right)\right\}
$$
and for $f\in\UPset$ and $F\in\CCC_f$ define
$$
  S(f,F)=\{z\in\Cset:\fmany n\in\Nset\ z\rest[f(n),f(n+1))\in F_n\}.
$$
It is easy to check that $S(f,F)\in\EE$ for all $f\in\UPset$ and $F\in\CCC_f$.
By \cite[Theorem 4.2]{MR1186905} (or see~\cite[2.6.3]{MR1350295}),
these sets actually form a base of $\EE$. We need a little more:
\begin{lem}\label{2.6.3}
$\forall E\in\EE\ \forall f\in\UPset\ \exists g\in\UPset\ \exists G\in\CCC_{f{\circ}g}\
E\subs S(f{\circ}g,G)$.
\end{lem}
\begin{proof}
Let $E\in\EE$. We may suppose that $E_n\upto E$ with $E_n$'s compact.
It is easy to show that if $C\subs\Cset$ is a compact null set, then
$$
\forall\eps>0\ \fmany n\ \exists T\subs2^n\quad
  C\subs\cyl T\ \&\ \frac{\abs T}{2^n}<\eps
$$
Therefore we may recursively define $g\in\UPset$ in such a way that
$g(n+1)>g(n)$ and
\begin{equation}\label{eq:C}
  \exists T_n\subs2^{f\circ g(n+1)}\quad  E_n\subs\cyl{T_n}\ \&\
  \frac{\abs{T_n}}{2^{f\circ g(n+1)}}<\frac{1}{4^{f\circ g(n)}}.
\end{equation}
Write $h=f\circ g$.
For $n\in\Nset$ define $G_n=\{s\rest [h(n),h(n+1)):s\in T_n\}$.
Obviously $\abs{G_n}\leq\abs{T_n}$. Therefore \eqref{eq:C} yields
$\frac{\abs{G_n}}{2^{h(n+1)-h(n)}}
\leq\frac{1}{2^{h(n)}}\leq\frac{1}{2^n}$.
Thus $\seq{G_n}\in\CCC_{h}$ and since $E_n\upto E$, we also have
$E\subs S(h,G)$, as desired.
\end{proof}
\begin{lem}\label{Einc3}
Let $f,g\in\UPset$, $F\in\CCC_f$ and $G\in\CCC_{f{\circ}g}$. Then
$S(f,F)\subs S(f{\circ}g,G)$ if and only if
\begin{equation}\label{Einc}
  \fmany n\in\Nset\ \forall k\in[g(n),g(n+1))\quad F_k\subs\{z\rest[f(k),f(k+1)):z\in G_n\}.
\end{equation}
\end{lem}
\begin{proof}
Suppose condition~\eqref{Einc} fails. Then there is $I\in[\Nset]^\Nset$ such that
\begin{equation}\label{Einc2}
\forall n\in I\ \exists k_n\in[g(n),g(n+1))\ \exists z_{k_n}\in F_{k_n}\
  \forall z\in G_n\ z_{k_n}\nsubseteq z.
\end{equation}
For each $k\notin\{k_n:n\in I\}$ choose $z_k\in F_k$ and let $z\in\Cset$ be a sequence that extends
simultaneously all $z_k$'s (including those defined in~\eqref{Einc2}).
Then obviously $z\in S(f,F)$.
On the other hand, condition~\eqref{Einc2} ensures that
$z\notin S(f{\circ}g,G)$.
Thus $S(f,F)\subs S(f{\circ}g,G)$ yields \eqref{Einc}.
The reverse implication is straightforward.
\end{proof}

\begin{prop}\label{EisEsharp}
If $X\subs\Cset$ is $\EE$-additive, then $X$ is sharply $\EE$-additive.
\end{prop}
\begin{proof}
Suppose $X$ is $\EE$-additive. Let $E\in\EE$.
We are looking for a \si compact set $\widetilde X\sups X$ such that
$\widetilde X+E\in\EE$.

There are $f\in\UPset$ and $F\in\CCC_f$ such that
$E\subs S(f,F)$. Since $S(f,F)\in\EE$, we have $X+S(f,F)\in\EE$. By Lemma~\ref{2.6.3} there are
$g$ and $G\in\CCC_{f{\circ}g}$
such that $X+S(f,F)\subs S(f{\circ}g,G)$, i.e.,
$x+S(f,F)\subs S(f{\circ}g,G)$ for all $x\in X$.

The set $\widetilde X$ we are looking for is
$$
  \widetilde X=\{x\in\Cset:x+S(f,F)\subs S(f{\circ}g,G)\}.
$$
Obviously $X\subs\widetilde X$. It is also obvious that
$\widetilde X+E\subs\widetilde X+S(f,F)\subs S(f{\circ}g,G)\in\EE$.
Thus it remains to show that $\widetilde X$ is $F_\sigma$.

For any $x\in\Cset$ and $k\in\Nset$ set
$$
  F_k^x=\{z+x\rest[f(k),f(k+1)):z\in F_k\}
$$
and consider the sequence $F^x=\seq{F_k^x}$.
Clearly $F^x\in\CCC_f$ and $S(f,F^x)=x+S(f,F)$.
Therefore $\widetilde X=\{x\in\Cset:S(f,F^x)\subs S(f{\circ}g,G)\}$.
Use Lemma~\ref{Einc3} to conclude that
$$
  x\in\widetilde X\Leftrightarrow
  \fmany n\in\Nset\ \forall k\in[g(n),g(n+1))\ F_k^x\subs\{z\rest[f(k),f(k+1)):z\in G_n\}.
$$
It follows that $\widetilde X$ is $F_\sigma$ as long as the sets
$$
  A_{n,k}=\{x\in\Cset:F_k^x\subs\{z\rest[f(k),f(k+1)):z\in G_n\}\}
$$
are closed for all $n$ and all $k\in[g(n,g(n+1))$.
Fix $n\in\Nset$ and $k\in[g(n),g(n+1))$. Decoding the definitions yields
$$
  x\in A_{n,k}\Leftrightarrow
  \exists y\in2^{[f(k),f(k+1))}\quad y\subs x\ \&\
  \forall z\in F_k\ \exists t\in G_n\
  z+y\subs t.
$$
Since the set
$\{y\in2^{[f(k),f(k+1))}:\forall z\in F_k\ \exists t\in G_n\ z+y\subs t\}$
is finite, the set $A_{n,k}$ is closed, as required.
We are done.
\end{proof}
\noindent
The proof of Theorems~\ref{mainME1} and~\ref{mainME2} is now complete.
Here are a few consequences.
First of them is Theorem~\ref{thm:ex2}:
$X\subs\Cset$ is $\MM$-additive if and only if it is $\EE$-additive.
\begin{coro}
Let $f:\Cset\to\Cset$ be a continuous mapping.
If $X\subs\Cset$ is $\MM$-additive, then so is $f(X)$.
\end{coro}
The following is a little surprising.
\begin{coro}
If $X\subs\Cset$ is $\EE$-additive, then $\phi(X\times E)\in\EE$
for each $E\in\EE$ and every Lipschitz mapping $\phi:\Cset\times\Cset\to\Cset$.
\end{coro}
\begin{proof}
Let $E\in\EE$.
Since $X$ is $\EE$-additive, it is \ssmz{}. By Theorem~\ref{prodUhnull}(iv),
$\uhm^1(X\times E)=0$. By Lemma~\ref{ulipschitz}(ii), $\uhm^1(\phi(X\times E))=0$,
i.e., $\phi(X\times E)\in\EE$.
\end{proof}
The notion of $\MM$-additive sets extends to finite cartesian powers of $\Cset$
in the obvious manner.
\begin{coro}\label{cpowers}
A set $X\subs(\Cset)^n$ is $\MM$-additive if and only if it is \ssmz{}.
\end{coro}
\begin{proof}
We provide the argument for $n=2$. Let $X\subs(\Cset)^2$.
Denote by $X_1, X_2$ the two projections of $X$.
An easy application of Kuratowski-Ulam Theorem proves that if $X$ is $\MM$-additive
in $(\Cset)^2$, then both $X_1$ and $X_2$ are $\MM$-additive in $\Cset$.
(\emph{Hint:} If $M\subs\Cset$ is meager, then $X+M\times\Cset$ is meager
in $(\Cset)^2$. Since $(X_1+M)\times\Cset=X+M\times\Cset$,
the set $X_1+M$ is meager, as required.)
Therefore they are \ssmz{} and by Theorem~\ref{productHUH}(i) $X_1\times X_2$
is \ssmz{}. A fortiori, $X$ is \ssmz{}.

Now suppose $X$ is \ssmz{}. Then both $X_1$ and $X_2$, being Lipschitz images of $X$,
are by Proposition~\ref{trivUH}(ii) also \ssmz{} and thus $\MM$-additive in $\Cset$.
By~\cite[Theorem 1]{MR2545840} a product of two $\MM$-additive sets in $\Cset$
is $\MM$-additive. Therefore $X_1\times X_2$ and a fortiori $X$ is $\MM$-additive.
\end{proof}

\section{Remarks}
\label{sec:rem}

\subsection*{\ssmz{} sets on the line}
It is clear how the notion of $\MM$-additive set extends to other topological
groups.
The addition operations on $\Cset$ and on the real line $\Rset$ are so different
that is was not understood for a long time if $\MM$-additive sets on $\Rset$
behave the same way as those on $\Cset$. Finally
Weiss~\cite{MR3241126} found the following solution.
Let $T:\Cset\to[0,1]$ be the standard mapping defined by
$T(x)=\sum_{n\in\Nset}2^{-n-1}x(n)$.
\begin{thm}[{\cite[1,10]{MR3241126}}]\label{weiss1}
A set $X\subs[0,1]$ is $\MM$-additive if and only $T^{-1}(X)$ is $\MM$-additive
in $\Cset$.
\end{thm}
A similar (and much easier) result holds for \ssmz{} sets:
\begin{prop}\label{weiss4}
A set $X\subs[0,1]$ is \ssmz{} if and only $T^{-1}(X)$ is \ssmz{}.
\end{prop}
\begin{proof}
\cite[Lemma 3.5]{MR2177439} asserts that for any set $U\subs[0,1]$
there are sets $U_0,U_1$ such that $U_0\cup U_1=T^{-1}(U)$ and
$\diam U_i\leq\diam U$ for both $i=0,1$.
It follows that $\uhm^h(T^{-1}(X))\leq 2\uhm^h(X)$ for every gauge $h$.
On the other hand, since $T$ is $1$-Lipschitz, $\uhm^h(X)\leq\uhm^h(T^{-1}(X))$
by Lemma~\ref{ulipschitz}(ii).

Use the two inequalities and Theorem~\ref{basicUhnull}(ii) to conclude the proof.
\end{proof}
\begin{thm}\label{weiss2}
A set $X\subs\Rset$ is $\MM$-additive if and only if it is \ssmz{}.
\end{thm}
\begin{proof}
Since both properties are \si additive, we may clearly suppose $X\subs[0,1]$.
The proof is a straightforward
application of Theorem~\ref{mainME1} and the above theorem and proposition.
\end{proof}
This theorem extends to $\Rset^n$:
By~\cite[11]{MR3241126} a product of two $\MM$-additive sets on $\Rset$
is $\MM$-additive on $\Rset^2$. Using this fact and the above Theorem~\ref{weiss2}
one can repeat the proof of Corollary~\ref{cpowers} to show:
\begin{thm}\label{weiss22}
A set $X\subs\Rset^n$ is $\MM$-additive if and only if it is \ssmz{}.
\end{thm}
\begin{coro}
Let $X\subs\Rset^2$. The following are equivalent.
\begin{enum}
\item $X$ is $\MM$-additive,
\item all orthogonal projections of $X$ on lines are $\MM$-additive,
\item at least two orthogonal projections of $X$ on lines are $\MM$-additive.
\end{enum}
\end{coro}

\subsection*{$\gamma$-sets}
Nowik and Weiss~\cite[Proposition 3.7]{MR1905154} prove that every $\gamma$-set
of reals is $\MM$-additive.
We will show a generalization of this result, namely that all
$\gamma$-sets are \ssmz{}.

Recall the notion of $\gamma$-set, as introduced by
Gerlits and Nagy~\cite{MR667661}.
A family $\mc U$ of open sets in a separable metric space $X$ is called
an \emph{$\omega$-cover} of $X$ if every finite subset of $X$ is
contained in some $U\in\mc U$.
A metric space $X$ is a \emph{$\gamma$-set} if every $\omega$-cover of $X$
contains a $\gamma$-cover.

\begin{prop}\label{gamma}
Every $\gamma$-set is \ssmz{}.
\end{prop}
\begin{proof}
Let $X$ be an infinite $\gamma$-set.
Let $\Seqeps$. We are looking for an $\seqeps$-fine $\gamma$-groupable cover.
For $n\in\Nset$ define $\del_n=\eps_{0+1+2+\dots+n}$.
Fix an infinite set $\{z_n:n\in\Nset\}\subs X$.
For $n\in\Nset$ and $F\in[X]^n$ put
$$
  F^\circ=\bigcup\nolimits_{x\in F}B\bigl(x,\tfrac12\del_n\bigr)
  \setminus\{z_n\}.
$$
The family $\{F^\circ:F\in[X]^{<\Nset}\}$
is obviously an $\omega$-cover. Therefore there is a sequence $\seq{F_k}$
of finite sets such that $\seq{F_k^\circ}$ is a $\gamma$-cover.
If $\abs F=n$, then $F^\circ$ misses $z_n$.
It follows that the cardinalities of
$F_k$'s are unbounded. Thus we may choose a subsequence $\seq{k_m}$
such that $\abs{F_{k_0}}<\abs{F_{k_1}}<\abs{F_{k_2}}<\dots$.
The sequence $\seq{F_{k_m}^\circ:m\in\Nset}$ is still a $\gamma$-cover.

Write $j_m=\abs{F_{k_m}}$.
Form a sequence $\seq{x_i}$ as follows: First enumerate all points in $F_{k_0}$,
then continue with points of $F_{k_1}$ and so on.
Note that if $x_i\in F_{k_m}$, then $i\leq j_0+j_1+\dots j_m\leq 0+1+\dots+j_m$
and thus $\eps_i\geq\del_{j_m}$.
Consequently $F_{k_m}^\circ\subs\bigcup\{B(x_i,\frac{\eps_i}{2}):x_i\in F_{k_m}\}$
and it follows that the families $\mc G_m=\{B(x_i,\frac{\eps_i}{2}):x_i\in F_{k_m}\}$
are witnessing that
$\seq{B(x_i,\frac{\eps_i}{2})}$ is a $\seq{\eps_i}$-fine
$\gamma$-groupable cover.
\end{proof}

\subsection*{Scheepers Theorem}
A metric space $X$ has the \emph{Hurewicz Property} if for any sequence
$\seq{\mc U_n}$ of open covers there are finite families $\mc F_n\subs\mc U_n$
such that, letting $F_n=\bigcup\mc F_n$, the sequence $\seq{F_n}$ is a
$\gamma$-cover of $X$.

Scheepers ~\cite[Theorem 1, Lemma 3]{MR1779763} proved that a product of a \smz{} set
and a \smz{} set with Hurewicz property is \smz{}.
It is easy to show that a \smz{} space with the Hurewicz property
is \ssmz{}.
Therefore Theorem~\ref{productHUH}(ii) improves Scheepers' result.
We claim that it is a proper extension: since \ssmz{} is a uniform property and
Hurewicz property is topological, one cannot expect \emph{a priori} that
every \ssmz{} set has the Hurewicz property.
(It is of course so if Borel Conjecture holds.)
There indeed is a CH example:
\begin{prop}
Assuming the Continuum Hypothesis, there is a \ssmz{} set
that does not have the Hurewicz property.
\end{prop}
\begin{proof}
It follows from~\cite[Theorem 1]{MR738943} and its proof that under
the Continuum Hypothesis there is a $\gamma$-set
$X\subs\Cset$ that is concentrated on a countable set $D$.
By Proposition~\ref{gamma}, $X$ is \ssmz{}. On the other hand,
as proved in~\cite[Theorem 20]{MR1610427}, the set $X\setminus D$ does
not have the Hurewicz property
and since it is a subset of $X$, it is \ssmz{}.
\end{proof}

\subsection*{Corazza's model}
Theorem~\ref{productHUH}(ii) also raises a question whether a space whose
product with any \smz{} set of reals is \smz{} has to be \ssmz{}.
The answer is consistently no.
A similar observation was noted without proof in~\cite{MR1610427}
and also in~\cite{MR3731016}.

We choose ``reals'' to refer to $\Cset$, since by Proposition~\ref{weiss4} it makes no
difference if we work in $\Cset$ or $\Rset$.
The following argument came out from a discussion with Tomasz Weiss.
Corazza~\cite{MR982239}
constructs a model of ZFC with the following properties.
Denote by $X$ the set of ground model reals.
\begin{enumerate}
\item[(a)] $X$ is not a meager set in the Corazza's extension.
\item[(b)] $\abs X=\omega_1$.
\item[(c)] A set of reals $Y$ in the extension is \smz{} if and only if $\abs Y\leq\omega_1$.
\end{enumerate}
By (a), $X$ is not $\MM$-additive and hence not \ssmz{}. By (b) and (c),
if $Y$ is any \smz{} set of reals, then $\abs{X\times Y}\leq\omega_1$.
Since $\Cset\times\Cset$ is uniformly homeomorphic to $\Cset$, $X\times Y$
is a uniformly continuous image of a set of cardinality at most $\omega_1$.
Such a set is \smz{} by (c) and thus $X\times Y$ is \smz{} as well.
We proved the following:
\begin{prop}\label{corazza}
In Corazza model there is a set $X\subs\Cset$ that is not \ssmz{}
and yet $X\times Y$ is \smz{} for each \smz{} set $Y\subs\Cset$.
\end{prop}
However, the following question remains unanswered:
\begin{question}
Is it consistent that there is a metric space $X$ such that $X\times Y$ is \smz{}
for every \smz{} metric space $Y$ and yet $X$ is not \ssmz{}?
\end{question}

\subsection*{Dimension inequalities}
By corollaries~\ref{hdimx} and~\ref{uhdimx}, if $Y$ is a \si compact metric space,
then $\hdim X\times Y=\hdim Y$ if $X$ is \smz{}, and $\uhdim X\times Y=\hdim Y$
if $X$ is \ssmz{}.

We note that the assumption of \si compactness imposed upon $Y$ cannot be dropped:
%
Tsaban and Weiss~\cite[Theorem 4]{MR2112732} construct,
under $\mathfrak p=\mathfrak c$, a $\gamma$-set $X$
(that is by Proposition~\ref{gamma} \ssmz) and a set $Y\subs\Rset$ such that
$\hdim Y=0$ and yet $\hdim X\times Y=1$.

\subsection*{Universally meager sets}
Recall that a separable metric space $E$ is termed \emph{universally meager}
\cite{MR1814112,MR2427418}
if for any perfect Polish spaces $Y,X$ such that $E\subs X$ and every
continuous one--to--one mapping $f:Y\to X$ the set $f^{-1}(E)$ is meager
in $Y$. We show that \ssmz{} sets are universally meager.
\begin{lem}\label{lemUM}
Let $X,Y,Z$ be perfect Polish spaces and $\phi:Y\to X$ a continuous one--to--one mapping.
Let $\mc F$ be an equicontinuous family of uniformly continuous mappings of $Z$ into $X$.
If $E\subs Z$ is \ssmz{}, then there is a \si compact set
$F\sups E$ such that $\phi^{-1}f(F)$ is meager in $Y$ for all $f\in\mc F$.
\end{lem}
\begin{proof}
Let $\{U_n\}$ be a countable base for $Y$.
As $\phi$ is one--to--one the set $\phi(U_n)$ is analytic and
uncountable for each $n$. Therefore it contains a perfect set and thus is not \ssmz{},
i.e., there is a gauge $h_n$ such that $\uhm^{h_n}(\phi U_n)>0$.
Choose a gauge $h$ such that $h\prec h_n$ for all $n$, so that
$\uhm^h(\phi U_n)>0$ for all $n$.
Therefore $\uhm^h(\phi U)>0$ for each nonempty set $U$ open in $Y$.

Since $\mc F$ is equicontinuous, there is a gauge $g$
such that~\eqref{lip2} is satisfied by each $f\in\mc F$.
By Theorem~\ref{basicUhnull} $E\in\NNs(\uhm_0^{h{\circ}g})$. Therefore there is a \si compact set
$F\sups E$ such that $\uhm^{h{\circ}g}(F)=0$. Hence Lemma~\ref{ulipschitz} guarantees that
$\uhm^h(f(F))=0$ for all $f\in\mc F$. Therefore the $F_\sigma$-set $\phi^{-1}f(F)$
is meager in $Y$: for otherwise
it would contain an open set witnessing $\uhm^h(f(F))>0$.
\end{proof}
Apply this lemma with $Z=X$ and $\mc F=\{\id_X\}$ to get

\begin{prop}\label{umg}
Every \ssmz{} set is universally meager.
\end{prop}


\subsection*{Meager additive sets in topological groups}
There is an obvious question: how far beyond $\Cset$ and $\Rset$ we can extend
the equivalence of $\MM$-additive and \ssmz{}.
\begin{question}
For what Polish groups are the notions of $\MM$-additive and \ssmz{} equivalent?
\end{question}

\subsection*{Null-additive sets}
A set $X\subs\Cset$ is termed \emph{null-additive} if for every Haar null set $N$
the set $X+N$ is Haar null.
In a follow-up of the present paper we will show that null-additive sets in $\Cset$
can be described in terms of packing measures and dimensions.

\section{Appendix: Hausdorff measures on cartesian products}

In this appendix, we prove a few integral inequalities needed for the proof
of Lemma~\ref{uhowroyd}. They generalize those proved by Howroyd
in his famous Thesis~\cite{howroydPhD} and Kelly~\cite{MR0318427}.

We need a notion of a weighted Hausdorff measure. Let $X$ be a metric space.
Say that a countable collection of pairs $\{(c_i,E_i):i\in I\}$ is a
\emph{weighted cover} of $E\subs X$ if $c_i>0$ and $E_i\subs X$ for all $i\in I$
and $\sum\{c_i:x\in E_i\}\geq 1$ for all $x\in E$.
We say it is $\del$-fine in the cover $\{E_i:i\in I\}$ is $\del$-fine,
i.e., if $\diam E_i\leq\del$ for all $i\in I$.

Let $g$ be a gauge and $E\subs X$. For each $\delta>0$ set
$$
  \whm^g_\delta(E)=
  \inf\left\{\sum_{n\in\Nset}c_i g(\diam E_i):
  \text{$\{(c_i,E_i)\}$ is a $\del$-fine weighted cover of $E$}\right\}
$$
and put
$
  \whm^g(E)=\sup_{\delta>0}\whm^g_\delta(E).
$

Properties of the weighted Hausdorff measures are discussed, e.g.,
in the two mentioned papers~\cite{howroydPhD,MR0318427}. Trivially
$\whm^g\leq\hm^g$. The converse inequality holds if $g$ satisfies
the doubling condition. It was proved in~\cite[2.10.24]{MR0257325}
for compact sets and in~\cite[9.8]{howroydPhD} in full generality.
\begin{thm}[{\cite[9.8]{howroydPhD}}]\label{HwH}
If $g$ is a doubling gauge, then $\whm^g=\hm^g$.
\end{thm}

The integrals in the following inequalities are the usual upper Lebesgue integrals:
If $\mu$ is a Borel measure on a metric space and $f:X\to[-\infty,\infty]$
a function, then
$$
 \upint f \,d\mu
 =\inf\left\{\int\phi \,d\mu:
 \phi\geq f\text{ Borel measurable}\right\}.
$$
Let $X,Y$ be metric spaces. Denote their respective metrics by $d_X$ and $d_Y$.
Recall that the cartesian product $X\times Y$ is equipped by the maximum metric
\eqref{maxmetric}.
For a set $E\subs X\times Y$ and $x\in X$, write $E_x$ for the vertical section
$\{y\in Y:(x,y)\in E\}$.

\begin{thm}\label{thm:appendix}
Let $X,Y$ be metric spaces and $E\subs X\times Y$. Let $g,h$ be gauges.
Then
\begin{enum}
\item $\upint\whm^g(E_x)\,d\whm^h(x)\leq\whm^{gh}(E)$,
\item $\upint\hm^g(E_x)\,d\whm^h(x)\leq\hm^{gh}(E)$.
\end{enum}
\end{thm}
\begin{proof}
First of all, we may assume that $E$ is Borel, since both Hausdorff
and weighted Hausdorff measures are Borel regular.
It is also routine to check that if $E$ is Borel, then the integrands
$x\mapsto\whm^g(E_x)$ and $x\mapsto\hm^g(E_x)$ are Borel measurable.
Therefore both integrals are standard Lebesgue integrals.

We prove (i) first and then indicate how to get (ii) by the same proof.
Approximating the integrand from below by a simple function and replacing
$X$ with the projection of $E$ onto the $x$-axis reduces (i)
to the following:
\begin{equation}\label{appendix}
\text{If $\whm^g(E_x)>\gamma$ for all $x\in X$, then $\whm^{gh}(E)\geq\gamma\whm^h(X)$.}
\end{equation}
Fix $\eps>0$. For every $x\in X$ there is $\del_x>0$ such that
$\whm^g_\del(E_x)>\gamma$. Also $\whm^h$ can be approximated by $\whm^h_\del$.
Therefore there is $\del>0$ such that
and a set $\hat X\subs X$ such that
\begin{enumerate}
\item[(a)] $\whm^g_\del(E_x)>\gamma$ for all $x\in\hat X$,
\item[(b)] $\whm^h_\del(\hat X)>\whm^h(X)-\eps$.
\end{enumerate}
Let $\mc C=\{(c_i,E_i):i\in I\}$ be a $\del$-fine weighted cover of $E$.
Denote by $p_X$ and $p_Y$ the respective projections.
For each $i\in I$ let $d_i=\frac{c_i}{\gamma}g(\diam(p_Y(E_i)))$
and consider the family $\mc D=\{(d_i,p_X(E_i)):i\in I\}$.

For each $x\in\hat X$ we have
$$
  \sum_{x\in p_X(E_i)}d_i=\frac1\gamma\sum_{x\in p_X(E_i)}c_ig(\diam p_Y(E_i))
  \geq\frac1\gamma\sum_{(E_i)_x\neq\emptyset}c_ig(\diam (E_i)_x)
$$
and since the family $\{(c_i,(E_i)_x):(E_i)_x\neq\emptyset\}$ is
obviously a weighted $\del$-fine cover of $E_x$, the latter sum is estimated
from below by $\whm^g_\del(E_x)$. Therefore (a) and the above calculation shows that
$\sum_{x\in p_X(E_i)}d_i\geq1$ for all $x\in\hat X$, i.e.,
that $\mc D$ is a weighted $\del$-fine cover of $\hat X$.
Therefore
\begin{align*}
\whm^h_\del(\hat X)
  &\leq\sum d_i h(\diam(p_X(E_i)))
  \leq\sum \frac{c_i}{\gamma} g(\diam(p_Y(E_i)))h(\diam(p_X(E_i)))\\
  &\leq\frac1\gamma\sum c_ig(\diam E_i)h(\diam E_i)
  =\frac1\gamma\sum c_i(gh)(\diam E_i).
\end{align*}
Multiplying with $\gamma$ and taking the infimum over all
$\del$-fine weighted covers of $E$ yields
$\gamma\whm^h_\del(\hat X)\leq\whm^{gh}_\del(E)$.
It thus follows from (b) that
$$
  \gamma(\whm^h(X)-\eps)\leq\whm^h_\del(\hat X)
  \leq \whm^{gh}_\del(E)\leq \whm^{gh}(E)
$$
and \eqref{appendix} obtains on letting $\eps\to0$.

(ii) follows from this proof simply by imposing an extra condition:
require that $c_i=1$ for all $i\in I$.
\end{proof}

Combining this theorem with Theorem~\ref{HwH} yields
\begin{coro}
If $h$ is a doubling gauge, then
$\upint\hm^g(E_x)\,d\hm^h(x)\leq\hm^{gh}(E)$.
\end{coro}
\begin{thm}
Let $X,Y$ be metric spaces and $E\subs X\times Y$. If $g,h$ are gauges,
then
\begin{enum}
\item
$\upint\uhm^g(E_x)\,d\whm^h(x)\leq\uhm^{gh}(E)$
\item
and if $g$ is doubling, then
$\upint\uhm^g(E_x)\,d\hm^h(x)\leq\uhm^{gh}(E)$.
\end{enum}
\end{thm}
\begin{proof}
The second inequality follows at once from the first one and Theorem~\ref{HwH}.
The first inequality obtains from Theorem~\ref{thm:appendix}(ii) as follows:
We may suppose that $X$ and $Y$ are both complete metric spaces.
Let $K\supseteq E$ be a \si compact set. By Lemma~\ref{lem1}(vi)
$\uhm^g(E_x)\leq\uhm^g(K_x)=\hm^g(K_x)$ for all $x$, hence~\ref{thm:appendix}(ii)
yields
$$
  \upint\uhm^g(E_x)\,d\whm^h(x)\leq\upint\hm^g(K_x)\,d\whm^h(x)
  \leq\hm^{gh}(K).
$$
Apply Lemma~\ref{lem1}(v) to conclude the proof.
\end{proof}
A particular choice of $E=X\times Y$ yields
Lemmas~\ref{howroyd} and~\ref{uhowroyd}.
%

\section*{Acknowledgments}
I would like to thank Michael Hru\v s\'ak for his kind support and
never-ending enlightening discussions while I was
visiting Instituto de matem\'aticas, Unidad Morelia,
Universidad Nacional Auton\'oma de M\'exico,
Peter Elia\v s for numerous valuable comments,
Tomasz Weiss for help with \ref{corazza}, and
Marion Scheepers and Boaz Tsaban who made me write the paper.
\bibliographystyle{amsplain}
\providecommand{\bysame}{\leavevmode\hbox to3em{\hrulefill}\thinspace}
\providecommand{\MR}{\relax\ifhmode\unskip\space\fi MR }
\providecommand{\MRhref}[2]{%
  \href{http://www.ams.org/mathscinet-getitem?mr=#1}{#2}
}
\providecommand{\href}[2]{#2}

\end{document}